\documentclass[11pt]{amsart}
\usepackage{amsmath, amssymb,amsthm, amsfonts}
\usepackage[margin=3.2 cm]{geometry}

\def\v{{\bf v}}
\def\r{\mathcal R}
\def\R{\mathbb R}

\def\C{\mathbb C} 

\def\F{\mathbb F}
\def\H{\mathbb H}

\def\z{{\bf z}}
\def\w{{\bf w}}

\def\V{\mathbb V}

\def\X{\mathbb X}
\def\A{\mathbb A}

\def \ch{{\bf H}_{\C}}

\def \h{{\bf H}_{\H}}
\def\fh{{\bf H}_{\mathbb F}}

\def\R{\mathbb R}
\def\V{\mathbb V}

\def\p{\mathbf p}

\newcommand{\SU}{\mathrm{SU}}
\newcommand{\GL}{\mathrm{GL}}

\newcommand{\Sp}{\mathrm{Sp}}

\def \a{\mathcal A}

\def\P{\mathbb P}

\def \ab {{\bf a}_B}
\def \aa {{\bf a}_A}
\def \ra {{\bf r}_A}
\def \rb {{\bf r}_B}

\def \xb{{\bf x}_{k, B}}
\def \yab {{\bf a}_{B'}}
\def \yaa {{\bf a}_{A'}}
\def \yra {{\bf r}_{A'}}
\def \yrb {{\bf r}_{B'}}

\def \yxb {{\bf x}_{k, B'}}

\def \a {{\bf a}}
\def \r  {{\bf r}}
\def \x  {{\bf x}}

\def \tr  {{\rm{tr}}}
\def \rtr  {{\rm{tr}}_{\mathbb R}}

\newcommand{\fX}{\mathfrak{X}}

\newcommand{\hm}{{\mathrm{Hom}}}

\newtheorem{theorem}{Theorem}[section]
\newtheorem{lemma}[theorem]{Lemma}

\theoremstyle{definition}
\newtheorem{definition}[theorem]{Definition}

\theoremstyle{remark}
\newtheorem{remark}[theorem]{Remark}
\numberwithin{equation}{section}
\theoremstyle{plain}
\newtheorem{acknowledgement}{Acknowledgement}

\newtheorem{corollary}[theorem]{Corollary}

\numberwithin{equation}{section}
\newcommand{\defref}[1]{Definition~\ref{#1}}
\newcommand{\secref}[1]{Section~\ref{#1}}
\newcommand{\thmref}[1]{Theorem~\ref{#1}}
\newcommand{\lemref}[1]{Lemma~\ref{#1}}

\newcommand{\corref}[1]{Corollary~\ref{#1}}
\newcommand{\eqnref}[1]{~{\textrm(\ref{#1})}}

\begin{document}

\title[Local coordinates for complex and quaternionic hyperbolic pairs]{Local coordinates for complex and quaternionic hyperbolic pairs}
 \author[Krishnendu Gongopadhyay  \and Sagar B. Kalane]{Krishnendu Gongopadhyay \and Sagar B. Kalane
 }
\address{Indian Institute of Science Education and Research (IISER) Mohali,
 Knowledge City,  Sector 81, S.A.S. Nagar 140306, Punjab, India}
\email{krishnendug@gmail.com, krishnendu@iisermohali.ac.in}
\address{Indian Institute of Science Education and Research (IISER) Pune, 
Dr. Homi Bhabha Road,
Pashan, Pune 411008, India} 
\email{sagark327@gmail.com}
 \thanks{Gongopadhyay acknowledges partial support from SERB-DST MATRICS project: MTR/2017/000355. Kalane is supported by  IISER Pune Institute post doctoral fellowship.}
\date{\today}
 \subjclass[2010]{Primary 37C15; Secondary  51M10, 15B57, 20E45}
\keywords{complex hyperbolic space, quaternionic hyperbolic space, loxodromic elements, character variety, surface group,  traces}
\begin{abstract}
Let $G(n)={\rm Sp}(n,1)$ or ${\rm SU}(n,1)$. We classify conjugation orbits of generic pairs of loxodromic elements in $G(n)$. Such pairs, called `non-singular', were introduced by Gongopadhyay and Parsad for ${\rm SU}(3,1)$. We extend this notion and classify $G(n)$-conjugation orbits of such elements in arbitrary dimension. For $n=3$, they give a subspace that can be parametrized using a set of coordinates whose local dimension equals the dimension of the underlying group. We further construct twist-bend parameters to glue such representations and obtain local parametrization for generic representations of the fundamental group of a closed (genus $g \geq 2$)  oriented surface into $G(3)$.
\end{abstract}
\maketitle

\section{Introduction}\label{intro}
  Let $\F=\H$ or $\C$, where  $\H$ denotes the division ring of Hamilton's quaternions. Let $G(n)$, or simply $G$, denote the group  $\SU(n, 1; \F)$ that acts as the isometry group of the $\F$-hyperbolic space $\fh^n$. Usually we denote $\SU(n, 1; \C)=\SU(n,1)$, and $\SU(n, 1; \H)=\Sp(n,1)$.  This paper concerns the problem of classifying $G$-conjugation orbits of loxodromic pairs in $G \times G$. The $G$-conjugation orbit space can be identified with the character variety or the deformation space   $\fX({\mathrm F}_2, G)=\hm({\mathrm F}_2, G)/ G$, where $G$ acts on $\hm({\mathrm F}_2, G)$ by inner automorphisms and  ${\mathrm F}_2=\langle x, y \rangle$ is the free group with generators $x$ and $y$.  In \cite{gk2}, we obtained  a  local parametrization of  a representation $\rho: {\mathrm F}_2 \to \Sp(n,1)$, where both $\rho(x)$ and $\rho(y)$ are semisimple.  When $G=\SU(n,1)$, for \hbox{loxodromic} pairs such a local parametrization is available from the work \cite{gp3}.  A main idea used in these works was to project fixed points of a pair of loxodromic elements onto the moduli space of $G$-congruence classes of ordered tuple of points on $\partial \fh^n$. Counting eigenvalues without multiplicities, a loxodromic element of $G$ has precisely two null eigenspaces and $n-1$ lines spanned by eigenvectors of positive norm. In \cite{gk2, gp3}, the $n-1$ lines spanned by these positive-definite eigenvectors were projected to the boundary $\partial \fh^n$. This associated tuple of points on $\partial \fh^n$ along with the spectrum data essentially classified the pair.   The difficulty to generalize the work from the complex hyperbolic isometries to the quaternionic hyperbolic set up arose due to the fact that the eigenvalues of an element in $\Sp(n,1)$ are not uniquely defined, but they appear in similarity classes.  So the conjugacy invariants available in $\Sp(n,1)$ are not well-behaved unlike their  complex counterpart. We avoided this difficulty by associating a combination of spatial and numerical invariants to obtain the local parametrizations in \cite{gk2}, \cite{gk1}.

Following the classical construction of the Fenchel-Nielsen coordinates on the Teichm\"uller space, especially for loxodromic representations in low dimensions,  one may like to have the local (real) dimension (or the `degrees of freedom') of the coordinates to add up to the  dimension of $\fX({\mathrm F}_2, G)$, which is the same as the (real) dimension of the Lie group $G$. We call such a parameter system as `Fenchel-Nielsen type'. The coordinate systems obtained in \cite{gk2} and \cite{gp3}, however, do not add up to the dimensions of the underlying group even for $n=3$. In general, it is unlikely to obtain a Fenchel-Nielsen type parameter system for arbitrary pairs as shown in \cite{gl}. However, it may be possible to associate Fenchel-Nielsen type coordinates (at least locally) to special subsets of the character variety. Parker and Platis obtained such a parameter system for  irreducible loxodromic representations $\fX({\mathrm F}_2,  \SU(2,1))$. In \cite{gk1}, we obtained Fenchel-Nielsen type coordinates for irreducible loxodromic  representations in 
$\fX({\mathrm F}_2, \Sp(2,1))$.    For generic loxodromic representations in $\fX({\mathrm F}_2, \SU(3,1))$, called `non-singular',  such a system of parameters is obtained from the work \cite{gp2}.  In \cite[Section 7.2]{gp3}, a version of non-singularity was defined for generic loxodromic pairs in $\SU(n,1)$. It  was proved that such a pair projects to a unique point on the moduli space of $\SU(n,1)$ congruence classes of ordered tuples of boundary points.

 In this paper, we have  extended the notion of non-singular pairs to  $\SU(n,1; \F)$ and have classified such pairs by associating a system of parameters. The associated numerical invariants  are comparable to the complex cross ratios used in  \cite{cugu1}. These invariants are obtained directly from the spectrum data of the pairs.  However, in the quaternionic setting, the quaternionic versions of the cross ratios are not enough to classify such pairs. A set of spatial parameters, called ``projective points", needs to be associated.  When one fixes the numerical invariants, these spatial parameters come from the fiber over the space of the numerical invariants.  This generalizes the parametrization obtained in \cite[Corollary  1.5]{gk1}. Though unlike the $\Sp(2,1)$ case, we do not know the precise domains of the numerical invariants. Restricting the  classification to $\SU(3,1; \F)$,  we obtain a Fenchel-Nielsen type parameter system for generic loxodromic representations in  $\fX({\mathrm F}_2, \SU(3,1; \F))$. As an application, we obtain local parametrization for generic representations of a closed genus $g$ surface group into $\Sp(3,1)$, where $g \geq 2$. This extends the work in \cite{gp2} over the quaternions.  

 Now, we define the `generic' representations which are investigated in this paper and  describe the results obtained.  Let $\F^{n,1}$ be the vector space $\F^{n+1}$ equipped with a non-degenerate Hermitian form $\langle.,. \rangle$ of signature $(n,1)$.  Then $\fh^n$ is the projectivization of the set of  vectors $v$ such that $\langle v, v \rangle<0$.  The boundary $\partial \fh^n$ is the projectivization of the null vectors. The projection of a vector $\v$ is denoted by $v$ on the projective space. A \emph{$k$-dimensional totally geodesic subspace} of $\fh^n$, that is also called an a \emph{$\F^k$-plane}, is the projectivization of a copy of $\F^{k,1}$ in $\F^{n,1}$. An $\F^1$-plane is simply called an \emph{$\F$-line}, and an $\F^{n-1}$-plane is simply called an \emph{$\F$-hyperplane}.   The boundary of an $\F^k$-plane is called an  \emph{$\F^k$-chain}.  
A point $v$ on the projective space is \emph{polar} to a $\F^{n-1}$-plane $C$ if the lift of $C$ in $\F^{n,1}$ is the orthogonal complement of $\v$. In particular, we must have $\langle \v, \v \rangle >0$. The positive vector $\v$ is polar to a $\F^{n-1}$-chain $L$ if $L$ is the boundary of a $\F^{n-1}$-plane $C$ that is polar to $v$.

 An element $A$ in $G$ is called hyperbolic (or loxodromic) if it has exactly two fixed points on $\partial \fh^n$. Such an $A$ has two eigenvalue classes represented by $re^{i \theta}$, $r^{-1} e^{i \theta}$, $r<1$, $\theta \in [-\pi, \pi]$, and rest of the $n-1$ classes are represented by $e^{i \phi_1}, \ldots, e^{i \phi_{n-1}}$, $\phi_i \in [-\pi, \pi]$.  An element $A$ in $G$ is \emph{regular} if the eigenvalue classes are mutually disjoint. 

Let $A$ be a  regular hyperbolic element. We denote by ${\bf a}_A$, ${\bf r}_A$ the null eigenvectors of $A$ corresponding to the classes $re^{i \theta}$ and $r^{-1}e^{i \theta}$ respectively.  Let ${ \bf x}_{j,A}$, $1 \leq j \leq n-1$, be the eigenvector to $e^{i \phi_j}$. The eigenvector ${ \bf x}_{j,A}$ is positive-definite, i.e. $\langle { \bf x}_{j,A}, { \bf x}_{j,A}\rangle>0$ for $1 \leq j \leq n-1$.  Note that $A$ fixes $x_{j, A}$ on $\F \P^n$.  For a hyperbolic (or loxodromic) element $A$ in $\SU(n,1)$, the characteristic polynomial determines the conjugacy class, and the traces $\tr (A^j)$,  $1 \leq j \leq [\frac{n+1}{2}]$, determine the coefficients of the characteristic polynomial. For $A \in \Sp(n,1)$, there is a natural complex representation $A_{\C}$ of $A$ in $\GL(2(n+1), \C)$. The tuple of the coefficients of the characteristic polynomial of $A_{\C}$ gives the \emph{real trace} of $A$, denoted by $\rtr(A)$.

In this paper we use the following:
\begin{definition} 
An element $A \in \Sp(n, 1)$ as \emph{loxodromic} if it is hyperbolic and having no real eigenvalue. 
\end{definition}
For a loxodromic $A$ in $\Sp(n, 1)$,  the real trace  $\rtr(A)$ is an element of $\R^{n+1}$. March\'e and Will in \cite{mw} have used flags in $\ch^2 \cup \partial \ch^2$ to give a set of local coordinates to generic elements on the  ${\rm PU}(2, 1)$ character variety of the fundamental group of a  punctured oriented surface. Taking motivation from their work, we use certain flags to define the generic pairs that we have investigated in this paper. 
\begin{definition} 
A \emph{flag} is a triple $(p, C, \Pi ),$ where $p$ is a point on $\Pi \cap \partial \fh^n$, $C$ is a $\F$-line containing $p$ on the boundary of $C$, $\Pi$ is a $\F$-hyperplane, and $C \subset \Pi$.

 Thus a positive point $x$  on $\F \P^n$ along with a boundary point $p$ and an $\F$-line $C$,  define a flag. 
\end{definition}

\begin{definition} Given a loxodromic element $A$, we  associate canonical flags to $A$ given by $F_{j, A}=(a_A, L_A, W_{j, A})$, $1 \leq j \leq n-1$, where $L_A$ is the line joining $a_A$ and $r_A$, and $W_{j, A}$ is the projectivization of $\x_{j, A}^{\perp}$.  \end{definition} 

\begin{definition}
Two flags $(p, C, \Pi)$ and $(p', C', \Pi')$ are said to form a  \emph{generic pair} if the following holds. 

\medskip (i) $p$ does not belong to the boundary of $C'$, $p'$ does not belong to the boundary of  $C$. 

\medskip (ii) $\partial C$ is disjoint from $\partial \Pi'$ and $\partial C'$ is disjoint from $\partial \Pi$. 
\end{definition} 
\begin{definition}\label{nsl} 
 Let $A$, $B$ be two loxodromic elements in ${\rm SU}(n,1; \F)$.  The pair $(A, B)$ is called \emph{weakly non-singular}  if 
 \begin{enumerate}
  \item { $A$ and $B$ does not have a common fixed point.   }

\item The elements $A$ and  $B$ are regular.

\item $(n-2)$ of the canonical flags of $A$ form generic pairs with $(n-2)$ of  the canonical flags of $B$. 
\end{enumerate} 
\end{definition} 

\begin{definition} \label{nsl1}
A pair $(A, B)$ of loxodromic elements in $\SU(n, 1; \F)$ is called \emph{non-singular} if it is weakly non-singular and the null fixed points of $A$ and $B$ do not belong to the boundary of the same proper totally geodesic hyperplane. We note that the last condition of non-singularity implies that $(A, B)$ is necessarily irreducible, i.e. $\langle A, B \rangle$ neither fixes a point, nor preserves a proper $\F^k$-plane. 
\end{definition} 

The above definition generalized the `non-singular' pairs defined in \cite{gp2}.  The terminology `non-singularity' in \cite{gp2} was motivated from the propoerty that the mixed cross ratios were non-zero for such a pair. Similar consideration are implicit in the above definition as well. 

 \medskip Corresponding to the boundary fixed points of $(A, B)$, we already have the conjugacy invariants given by the cross ratios and the angular invariants.  We recall here that for four distinct points $z_1$, $z_2$, $z_3$ and $z_4$ in $\partial \fh^n$, the \emph{usual cross-ratio} is defined by: 
\begin{equation}\label{ecr} \X(z_1, z_2, z_3, z_4)={\langle {\bf z}_3, {\bf z}_1 \rangle \langle {\bf z}_3, \bf z_2 \rangle}^{-1} { \langle {\bf z}_4, {\bf z}_2\rangle \langle   {\bf z}_4, {\bf z}_1 \rangle^{-1}},\end{equation}
where ${\bf z_i}$ is a lift of $z_i$ in $\F^{n,1}$. These cross ratios were introduced by Kor\'anyi and Reimann for points on $\partial \ch^n$ in \cite{kr}, also see \cite{gold}. Platis has investigated quaternionic versions of these cross ratios in \cite{platis}. The complex cross ratios are independent of the chosen lifts of $z_i$ and are conjugacy invariants. However, the quaternionic cross ratios are not independent of the chosen lifts of the points; therefore, they are not well-defined conjugacy invariants.  But similarity classes of the cross ratios are independent of the chosen lifts.  Accordingly,  $\Re(\X)$ and $|\X|$ are the conjugacy invariants associated to the quaternionic cross ratios. Also unlike the complex case, quaternionic cross ratios do not classify a quadruple of boundary points up to $\Sp(n,1)$-congruence.

It can be seen that modulo the symmetric group action on the four boundary fixed points of $(A, B)$, only three such cross ratios are needed to determine the others under the permutation. We denote these cross ratios by:
$$\X_1(A, B)=\X(a_A, r_A, a_B, r_B), ~\X_2(A, B)=\X(a_A, r_B, a_B, r_B), ~\X_3(A, B)=\X(r_A, r_B, a_B, a_A).$$
 Platis proved in \cite{platis} that for $n \geq 3$, the set of  cross ratios $(\X_1, \X_2, \X_3)$ of quadruple of points on $\partial \fh^n$ form a five dimensional semi-algebraic subset of $\R^5$. 
 
In the quaternionic set up, the Cartan's angular invariant associated to a triple $(z_1, z_2, z_3)$ on $\h^n \cup \partial \h^n$ is given by the following, see \cite{ak}, \cite{cao}, 

\begin{equation} \label{ainv} \A(z_1,~z_2,~z_3)=\arccos\frac{\Re(-\langle \z_1, \z_2, \z_3\rangle)}{|\langle \z_1, \z_2, \z_3\rangle|}, \end{equation} 
where  $\langle \z_1, \z_2, \z_3\rangle=\langle \z_1, \z_2\rangle\langle \z_2, \z_3\rangle\langle \z_3, \z_1\rangle$. The quaternionic angular invariants are independent of the chosen lifts of $z_i$ and are conjugacy invariants. So, there are angular invariants that correspond to the quadruple of the boundary fixed points.  We denote these angular invariants by    
$$\A_1(A, B)=\A(a_A, r_A, a_B), ~\A_2(A, B)=\A(a_A, r_A, r_B), ~ \A_3(A, B)=\A(r_A, a_B, r_B).$$
 In \cite{cao}, Cao proved that an ordered quadruple of  points on $\partial \h^n$ is determined up to $\Sp(n,1)$ congruence by the similarity classes of the cross ratios and the above angular invariants. 

In order to classify a weakly non-singular pair $(A, B)$, we would require more invariants. For this,  we extend the above definition of the cross ratio by taking one (or more) of the points $z_i$ to be points on $\F \P^n$ corresponding to the  positive definite eigenvectors of $A$ and $B$. We call such invariants as \emph{generalized cross ratios}. 
We also define generalized Goldman's eta invariants that corresponds to two boundary points and a hyperplane, see \cite[Section 7.3.1]{gold}.  The set of numerical invariants considered here comes from the Gram matrix associated to the pair $(A, B)$. For $(A, B)$ in $\Sp(n,1)$,  it is the similarity classes of these numerical quantities which are conjugacy invariants. So, the real parts and the moduli of the quantities are the conjugacy invariants associated to the $\Sp(n,1)$ conjugation orbit of $(A, B)$. However, these numerical invariants do not classify the pair $(A, B)$ completely. Rather, there is a whole fiber of points that corresponds to a fixed tuple of numerical invariants. These fibered elements correspond to the product of copies of $\C \P^1$ that we  call as \emph{projective points}  of $(A, B)$. Each of these $\C \P^1$ represents an eigenspace of $A$ or $B$, and a point on the given $\C \P^1$  corresponds to an `eigenset'. We note here that corresponding to a regular loxodromic, there are $n$ projective points, one each for the $n-1$ space-like eigenvectors, and one for the null eigenvectors.   With these terminologies, we have the following theorem where we refer to \secref{ci} for the precise list of the numerical invariants mentioned here.  

\begin{theorem} \label{thm2} Let $\rho: {\mathrm F}_2 \to \Sp(n,1)$ be a representation such that $(\rho(x), \rho(y))$ is weakly non-singular. Then $\rho$ is determined uniquely in the character variety by the $\rtr(\rho(x))$, $\rtr(\rho(y))$, the angular invariant $\A(a_{\rho(x)}, r_{\rho(x)}, a_{\rho(y)})$, the projective points and the $\Sp(1)$-conjugation orbit of the (unordered) tuple consisting of the usual cross ratios, the generalized cross ratios and the Goldman's eta invariants.  
\end{theorem}
The following theorem follows by restricting the proof of the above theorem over complex numbers.  
\begin{theorem} \label{cor3} Let $\rho: {\mathrm F}_2 \to \SU(n,1)$ be a representation such that $(\rho(x), \rho(y))$ is weakly non-singular. Then $\rho$ is determined uniquely in the character variety by  $\tr(\rho(x)^j)$, $\tr(\rho(y)^j)$, $1 \leq j \leq [\frac{n+1}{2}]$,  the angular invariant $\A(a_{\rho(x)}, r_{\rho(x)}, a_{\rho(y)})$,  the usual cross ratios,  the generalized cross ratios and the Goldman's eta invariants. 
\end{theorem}
\thmref{cor3} is implicit in the work \cite{gp3} and the above statement was noted in an older version: arXiv 1705.10469v2.

\medskip However, the degrees of freedom of the parameters in the above classification do not add up to the dimension of the group even in the lower dimensions. We would like to further obtain a smaller subfamily of invariants that might be sufficient for the classification. First, we shall consider the group $\SU(n,1)$.  In the following, we have used a method that is similar to the one used in \cite{gp2}. We would only need the following generalized cross ratios to classify a non-singular pair. For $1 \leq k \leq n-2$, let 
 $$\alpha_k(A, B)=\X(a_A, r_A, a_B, x_{k, B}),  ~\beta_k (A, B)=\X(a_B, r_B, a_A, x_{k, A}).$$
By the definition of non-singularity, the above quantities are non-zero and well-defined. 
In the case of $\SU(n,1)$, Cunha and Gusevskii proved in \cite{cugu2} that the moduli of ordered quadruple of points $(p_1, p_2, p_3, p_4)$ on $\partial \ch^n$ is determined by a point on a five dimensional subspace of $\R^5$  that consists of the points $(\A(p_1, p_2, p_3), \X_1(p_1, p_2, p_3, p_4), \X(p_1, p_4, p_3, p_2))$ satisfying some semi-algebraic equation. We shall use a point on this `Cunha-Gusevskii variety'. We have the following result in this set up that generalizes \cite[Theorem 1.1]{gp2}. 
\begin{theorem} \label{cor2}
Let $(A, B)$ be a non-singular pair in $\SU(n,1)$.  Then the $\SU(n, 1)$ conjugation orbit of $(A, B)$ is 
uniquely determined by the following parameters:
\medskip 
\begin{itemize}
\item [$\bullet$] $\tr(A^j),~\tr(B^j)$,  $1 \leq j \leq [\frac{n+1}{2}]$,
\item[$\bullet$]the cross ratios $\X_k(A, B)$, $k=1,2$,
\item[$\bullet$] the angular invariant $\A(a_A, r_A, a_B)$,
\item[$\bullet$]the $\alpha$-invariants $\alpha_k(A, B)$ and the $\beta$-invariants $\beta_k(A,B)$, $1 \leq k \leq n-2$. 
\end{itemize}
\end{theorem}

Restating the above theorem in terms of representations, we have the following. 
\begin{theorem} \label{cns}
Let $\rho: {\rm F}_2 \to {\rm SU}(n,1)$ be a representation such that $(\rho(x), \rho(y))$ is non-singular. Then the point $\rho$ in $\fX({\mathrm F}_2, \SU(n,1))$ is uniquely determined by the following parameters:
\medskip 
\begin{itemize}
	\item [$\bullet$]$\tr(\rho(x)^j),~\tr(\rho(y)^j)$,  $1 \leq j \leq [\frac{n+1}{2}]$,
	\item [$\bullet$]the cross ratios $\X_k(\rho(x), \rho(y))$, $k=1,2$, 
	\item [$\bullet$] the \hbox{angular} invariant $\A(a_{\rho(x)}, r_{\rho(x)}, a_{\rho(y)})$,
	\item [$\bullet$]the $\alpha$-invariants $\alpha_k(\rho(x), \rho(y))$ and  the $\beta$-invariants $\beta_k(\rho(x), \rho(y))$, $1 \leq k \leq n-2$.
\end{itemize}	   
\end{theorem} 

Thus, the local dimension of the coordinates adds up to at most $6n-1$: for the traces at most $n+1$ contributing at most $2n+2$; for the point on the cross ratio variety $5$; for the ~$\alpha$ and $\beta$-invariants $4(n-2)= 2 \times (n-2)$-$\alpha$ invariants + $2 \times (n-2)$-$\beta$-invariants;  the total adds up to $2n+2+5+4(n-2)=6n-1$. 

 A particularly interesting case appears when $n=3$.  In this case, the traces of loxodromics form a real three dimensional family, and the above parameters add up to $15$, the dimension of $\SU(3,1)$. 
\begin{corollary} \cite[Theorem 1.1]{gp2} 
Let $\rho: {\rm F}_2 \to {\rm SU}(3,1)$ be a representation such that $(\rho(x), \rho(y))$ is non-singular.  Then the point $\rho$ in $\fX({\mathrm F}_2, \SU(n,1))$ is uniquely determined by the following $15$ dimension  parameter system.

\begin{itemize}
	\item [$\bullet$] $\tr(\rho(x)),~\tr(\rho(y))$,
	\item[$\bullet$] $\sigma(\rho(x))$, $\sigma(\rho(y))$,
	\item[$\bullet$] $\X_k(\rho(x), \rho(y))$, $k=1,2,3$,
	\item[$\bullet$] $\alpha_1(\rho(x),\rho(y))$, $\beta_1(\rho(x), \rho(y))$,
\end{itemize}
where for an element $g\in \SU(3,1)$,  $\sigma(g)=(\tr^2(g)-\tr(g^2))/2$.
\end{corollary}

However, for $n \geq 4$, the local dimensions of the above parameter system is lesser than the dimension of the underlying group. With larger $n$, the upper bound $6n-1$ of the dimension of the parameter system becomes smaller in comparison to the dimension of $\SU(n,1)$ which is $n^2+2n$. 

\medskip Now we shall consider the quaternionic case. 
 An advantage of \thmref{thm2} is that the numerical invariants used there do not depend on the choices of the lifts of points of $\H \P^n$ to $\H^{n,1}$, and they serve as well-defined conjugacy invariants. But the similarity classes of $\alpha_k(A, B)$ and $\beta_k(A, B)$ do not determine the Gram matrix of $(A, B)$ uniquely.  This calls for some adjustment in the choices of the invariants. One way to avoid this difficulty is to adopt the convention of fixing a frame of reference. We adopt the convention of fixing the lift of the attracting fixed points. We shall take the standard lift,   see \secref{prel}, of the attracting fixed point of $A$ in the pair $(A, B)$. After this restriction, the numerical quantities $\alpha_k(A, B)$ and $\beta_k(A, B)$ will be well-defined invariants, as well as the usual cross ratios will be uniquely assigned to $(A, B)$. Comparable convention of fixing a frame of reference was used by Jiang and Gou in \cite{jg} in their understanding of the moduli space of ordered quadruples on $\partial \h^n$. In view of the chosen frame of reference, we have the following.
\begin{theorem} \label{qc2}
Let $\rho: {\rm F}_2 \to {\rm Sp}(n,1)$ be a representation such that $(\rho(x), \rho(y))$ is non-singular. We adopt the convention of taking the standard lift of the fixed point $a_{\rho(x)}$.  Then the point $\rho$ in $\fX({\mathrm F}_2, \Sp(n,1))$ is  determined by the following parameters: 

\begin{itemize}
	\item [$\bullet$]  $\rtr(\rho(x)),~\rtr(\rho(y))$,
	\item[$\bullet$] the angular invariants $\A_k(\rho(x), \rho(y))$,
	\item[$\bullet$] the usual cross {ratios} $\X_k(\rho(x), \rho(y))$, k=1,~2,~3,
	\item[$\bullet$] the $\alpha$-invariants $\alpha_k(\rho(x), \rho(y))$,  the $\beta$-invariants $\beta_k(\rho(x), \rho(y))$, {$1 \leq k \leq n-2$},
	\item[$\bullet$] the projective points 
	$(p_1(\rho(x)), \ldots, p_n(\rho(x))),~ (p_1(\rho(y)),\ldots, p_n(\rho(y)))$. 
\end{itemize}
\end{theorem} 
The degrees of freedom of the above set of coordinates add up to at most $14n-6$ (for each real traces $n+1$, contributing $2 \times (n+1)=2(n+1)$; for the point on the cross ratio variety $5$; for three angular invariants $3$;  for the projective points $4n=2 \times (2n \hbox{ projective points})$);  for the ~$\alpha$ and $\beta$-invariants: $8(n-2)=2 \times 4(n-2)$.  For $n=3$, the degrees of freedom add up to 36, which is the dimension of $\Sp(3,1)$. 
\begin{corollary}\label{qcc2}
Let $\rho: {\rm F}_2 \to {\rm Sp}(3,1)$ be a representation such that $(\rho(x), \rho(y))$ is non-singular.  Then the point $\rho$ in $\fX({\mathrm F}_2, \Sp(3,1))$ is  determined by the following parameters:

\begin{itemize}
	\item [$\bullet$] $\rtr(\rho(x)),~\rtr(\rho(y))$; for k=1,~2,~3,
	\item[$\bullet$]  the angular invariants $\A_k(\rho(x), \rho(y))$,
	\item[$\bullet$] the usual cross ratios  $\X_k(\rho(x), \rho(y))$;
	\item[$\bullet$]  $\alpha_1(\rho(x), \rho(y))$,  $\beta_1(\rho(x), \rho(y))$,
	\item[$\bullet$]the projective points 
	$(p_1(\rho(x)), p_2(\rho(x)), p_3(\rho(x)))$, $(p_1(\rho(y)), p_2(\rho(y)), p_3(\rho(y)))$.  
\end{itemize}
\end{corollary}

This motivates us to construct a gluing process to glue such a representation and associate  coordinates to generic surface group representations into $\Sp(3,1)$. Let $\Sigma_g$ denote a closed, connected, orientable surface of genus $g \geq 2$. Let $\pi_1(\Sigma_g)$ denote the fundamental group of $\Sigma_g$. Choose $\mathcal C= \{\gamma_j\}$, $j=1,2,\ldots, 3g-3$, a maximal family of simple closed curves on $\Sigma_g$ such that no two of $\gamma_j$ are neither homotopically equivalent, nor homotopically \hbox{trivial}.  The homotopy type of the curves may be considered to be elements of $\pi_1(\Sigma_g)$. We also assume that $g$ of the curves $\gamma_j$ correspond to two boundary components of the same three-holed sphere. Consider discrete, faithful representations $\rho: \pi_1(\Sigma_g) \to {\rm SU}(3,1; \F)$ such that the $3g-3$ group elements $\rho(\gamma_j)$ are loxodromics and each of the groups $\langle \rho(\gamma_k), \rho(\gamma_l) \rangle$ obtained from the given decomposition is  non-singular.    We call such a representation as \emph{non-singular}. We construct `twist-bend' parameters to glue such representations. Complex hyperbolic twist bends for representations into $\SU(3,1)$ were constructed in \cite{gp2}.  However, the method in \cite{gp2} does not generalize to $\Sp(3,1)$. Here, we generalize the approach used in \cite{gk1} to construct the twist-bend parameters. We have noted the construction for representations into $\Sp(3,1)$ for emphasizing the quaternionic hyperbolic case. The same method restricts to $\SU(3,1)$ as well, thus providing an alternative approach to the construction of twist bends in the complex hyperbolic case. Then using standard arguments as in  \cite{pp} or \cite{gk2}, we have the following result. 

\begin{theorem}\label{mth2} 
Let $\Sigma_g$ be a closed orientable surface of genus $g \geq 2$ with a simple curve system $\mathcal C= \{\gamma_j\}$, $j=1,2,\ldots, 3g-3$. Let $\rho: \pi_1(\Sigma_g) \to {\rm Sp}(3,1)$ be a non-singular representation of the surface group $\pi_1(\Sigma_g)$ into ${\rm Sp}(3,1)$. There are $72g-72$ real parameters that determine  $\rho$ in the character variety  ${\rm  Hom}(\pi_1(\Sigma_g), {\rm Sp}(3,1))/{\rm Sp}(3,1)$.  
\end{theorem}

When considering the representations into $\SU(3,1)$, we recover \cite[Thorem 1.3]{gp2}.

\begin{theorem}\label{mth2}
For $g \geq 2$, let $\Sigma_g$ be a closed orientable surface of genus $g$ with a simple curve system $\mathcal C= \{\gamma_j\}$, $j=1,2,\ldots, 3g-3$. Let $\rho: \pi_1(\Sigma_g) \to {\rm SU}(3,1)$ be a non-singular representation of the surface group $\pi_1(\Sigma_g)$ into ${\rm SU}(3,1)$. There are $30g-30$ real parameters that determine  $\rho$ in the character variety  ${\rm  Hom}(\pi_1(\Sigma_g), {\rm SU}(3,1))/{\rm SU}(3,1)$.  
\end{theorem}

\subsubsection*{Structure of the paper} In \secref{prel}, we briefly recall basic notions and notations. We follow similar notations as in our early papers \cite{gk1} or \cite{gk2}. We recall and re-interpret the projective points in \secref{pp}. In \secref{ns}, we prove \thmref{thm2}. In \secref{mnth}, we prove \thmref{cor2} and \thmref{qc2}. The twist-bend parameters are constructed in \secref{twb} and a sketch of the proof of \thmref{mth2} is given in \secref{twb}. 

\section{Preliminaries}\label{prel}
\subsection{Matrices over the quaternions} 
Let $\V$ be a right vector space over $\H$ and $T$ be a right linear transformation of $\V$. After choosing a basis of $\V$, such a linear transformation can be represented with a $n \times n$ matrix $M_T$ over $\H$, where $n=\dim \V$. The map $T$ is invertible if and only if $M_T$ is invertible.  Suppose $\lambda\in \H^{\ast}$ is a (right) eigenvalue of $T$. Let $v$ be an eigenvector to $\lambda$. Note that for $\mu \in \H^{\ast}$,
$$T(v \mu)=T(v) \mu=(v \lambda )\mu=(v \mu) \mu^{-1} \lambda \mu.$$ Thus, the eigenvalues of $T$ occur in similarity classes and if $v$ is a $\lambda$-eigenvector, then $v \mu \in v \H$ is a $\mu^{-1} \lambda \mu$-eigenvector.
Thus the eigenspace $v \H$ is not uniquely assigned to a single eigenvalue, but to the similarity class of $\lambda$. So the similarity classes of eigenvalues are conjugacy invariants over the quaternions, and notion of characteristic or minimal polynomial is not well-defined. Each similarity class of eigenvalues contains a unique pair of complex conjugate numbers. We shall choose one of these complex numbers $re^{i \theta}$, $\theta \in [0, \pi]$,  to be the representative of its similarity class. We may refer a similarity class representative as `the eigenvalue of $T$', though it should be understood that our reference is towards the similarity class. At places, where we need to distinguish between the similarity class and a representative, we shall denote the similarity class of an eigenvalue representative $\lambda$ by $[\lambda]$. 

\subsection{The hyperbolic space} Let $\F= \H$ or $\C$. 
Let $\V=\F^{n,1}$ be the $n$-dimensional right vector space over $\F$ equipped with the Hermitian form of signature $(n,1)$ given by $$\langle\z,\w\rangle=\w^{\ast}H\z=\bar w_{n+1}z_{1}+\bar w_2 z_2+\cdots+ \bar w_{n} z_{n}+\bar w_1 z_{n+1},$$
where $\ast$ denotes conjugate transpose. The matrix of the Hermitian form is given by
\begin{center}
$H=\left[ \begin{array}{cccc}
            0 & 0 & 1\\
           0 & I_{n-1} & 0 \\
 1 & 0 & 0\\
          \end{array}\right],$
\end{center}
where $I_{n-1}$ is the identity matrix of rank $n-1$.
We consider the following subspaces of $\H^{n,1}:$
$$\V_{-}=\{\z\in \F^{n,1}:\langle\z,\z \rangle<0\}, ~ \V_+=\{\z\in\F^{n,1}:\langle\z,\z \rangle>0\},$$
$$\V_{0}=\{\z \in \F^{n,1}\setminus\{{\bf 0}\}:\langle\z,\z \rangle=0\}.$$
A vector $\z$ in $\F^{n,1}$ is called \emph{positive, negative}  or \emph{null}  depending on whether $\z$ belongs to $\V_+$,   $\V_-$ or  $\V_0$. Let $\P:\F^{n,1}\setminus\{{\bf 0}\}\longrightarrow  \F \P^n$ be the right projection onto the quaternionic projective space. Image of a vector $\z$ will be denoted by $z$.  The quaternionic hyperbolic space $\fh^n$ is defined to be $\P(\V_{-})$. The ideal boundary $\partial\fh^n$  is defined to be $\P (\V_{0})$. So we can write $\fh^n=\P(\V_{-})$ as
$$\fh^n=\{(w_1,\ldots, w_n)\in\H^n \ : \ 2\Re(w_1)+|w_2|^2+\cdots+|w_n|^2<0\},$$
where for a point $\z=\begin{bmatrix}z_1 & z_2 & \ldots & z_{n+1}\end{bmatrix}^T \in \V_- \cup \V_0$, $w_i=z_i z_{n+1}^{-1}$ for $i=1, \ldots, n$. This is the Siegel domain model of $\fh^n$. Similarly one can define the ball model by replacing $H$ with an equivalent Hermitian form $H'$ given by the diagonal matrix: $H'=diag(-1,1, \ldots, 1)$. 
We shall mostly use the Siegel domain model here.

There are two distinguished points in $\V_{0}$ which we denote by  $\bf{o}$ and $\bf\infty,$ given by
$$\bf{o}=\left[\begin{array}{c}
               0\\0\\ \vdots \\ 1\\
              \end{array}\right],
~~ \infty=\left[\begin{array}{c}
               1\\0 \\ \vdots \\0\\
              \end{array}\right].$$\\
Then we can write $\partial\h^n=\P(\V_{0})$ as
$$\partial\fh^n\setminus\{\infty\}=\{(z_1,\ldots,z_n)\in\H^n:2\Re(z_1)+|z_2|^2+\cdots+|z_n|^2=0\}.$$\\
Note that $\overline{\fh^n}=\fh^n \cup \partial \fh^n$. 

 Given a point $z$ of $\overline{\fh^n}\setminus\{\infty\} \subset\F \P^n$ we may lift $z=(z_1,\ldots, z_n)$ to a point $\z$ in $\V$, called the \emph{standard lift} of $z$. It is represented in projective coordinates by
 $$\z=\left[\begin{array}{c}
                z_1\\ \vdots \\ z_n\\1\\
               \end{array}\right].$$
 The Bergman metric in $\fh^n$ is defined in terms of the Hermitian form given by:
$${ds}^2=-\frac{4}{\langle \z,\z \rangle^2} \det \left[\begin{array}{cc}
                                                    \langle \z,\z \rangle & \langle d\z ,\z \rangle\\
                                                    \langle \z,d\z \rangle & \langle d\z,d\z \rangle\\
                                                   \end{array}\right].$$
If $z$ and $w$ in $\fh^n$ correspond to vectors $\z$ and $\w$ in $\V_{-}$, then the Bergman metric is also given by the distance $\rho$:
$$\cosh^2\bigg(\frac{\rho(z,w)}{2}\bigg)=\frac{\langle\z,\w\rangle \langle\w,\z\rangle}{\langle\z,\z\rangle \langle\w,\w\rangle}.$$

More information on the basic formalism of the quaternionic hyperbolic space may be found in \cite{cg}. 

\subsection{Isometries}   \label{ltr}
 Let ${\rm U}(n,1; \F)$ be the isometry group of  the Hermitian form $\langle .,.\rangle$. Each matrix $A$ in ${\rm U}(n,1; \F)$ satisfies the relation $A^{-1}=
H^{-1}A^{\ast}H$, where $A^{\ast}$ is the conjugate transpose of $A$. The isometry group of  $\fh^n$ is the projective unitary group ${\rm PU}(n,1; \F)$, the group ${\rm U }(n,1)$ modulo the center. We denote ${\rm U}(n,1; \C)={\rm U}(n,1)$, and ${\rm U}(n,1; \H)=\Sp(n,1)$. 
\subsection{Hyperbolic elements in $\SU(n,1;\F)$} Let $A$ be hyperbolic in $\SU(n,1;\F)$.  
 Let $a_A \in \partial\fh^n$ be the \emph{ attracting fixed point } of $A$ that corresponds to the eigenvalue $re^{i \theta}$, $r<1$,  and let $r_A \in \partial\fh^n$ be the \emph{repelling fixed point} corresponding to the eigenvalue $r^{-1}e^{i \theta}$. Let $a_A$ and $r_A$ lift to eigenvectors $\a_A$ and $\r_A$ respectively.  Let $\x_{j,A}$ be an eigenvector corresponding to $e^{i \phi_j}$, $j=1, \dots, n-1$.  The points $x_{j,A}$, $j=1, \ldots, n-1$ on $\P(\V_+)$ are  the space-like (or positive-definite) projective fixed points of $A$. Define $ E_A(r, \theta, \phi_1, \ldots, \phi_{n-1})$ as
\begin{equation}\label{li1}
 E_A(r, \theta,\phi_1, \ldots, \phi_{n-1})=\begin{bmatrix} re^{i \theta} & 0 & \ldots  & 0 &0  \\ 0 & e^{i \phi_1} & \ldots & 0  & 0\\ & &  \ddots & & & \\ 0 & 0 & \ldots & e^{i \phi_{n-1}}& 0 \\ 0 & 0 & \ldots & 0 & r^{-1} e^{ i \theta} \end{bmatrix} 
\end{equation}
 Let $C_A=\begin{bmatrix}   \a_A & \x_{1,A} & \ldots & \x_{n-1,A} & \r_A
  \end{bmatrix}$ be the  matrix corresponding to the eigenvectors. We can choose $C_A$ to be an element of $\Sp(n,1)$ by normalizing the eigenvectors:
$$\langle \a_A, \r_A \rangle=1,~\langle \textbf{x}_{i, A},\textbf{x}_{i,A} \rangle=1, \hbox{ } i=1, \ldots, n-1.$$
 Then $A=C_A E_A(r, \theta, \phi_1,\ldots, \phi_{n-1}) C_A^{-1}$. 
\begin{lemma}\cite{cg}\label{hycl}  {\rm (Chen-Greenberg )}
 Two hyperbolic elements in $\SU(n,1; \F)$  are conjugate if and only if they have the same similarity classes of eigenvalues.
\end{lemma}

\begin{definition} 
 Let $A$ be a hyperbolic element in $\SU(n,1;\F)$. Let $\lambda$ represents an eigenvalue from the similarity class of eigenvalues $[\lambda]$ of $A$. Let $\x$ be a $\lambda$-eigenvector. Then $\x$ defines a point $x$  on $\F\P^n$ that is either a point on  $\partial \fh^n$ or, a point in $\P(\V_{+})$. The lift of $x$ in $\F^{n,1}$ is the quaternionic line $\x \F$. We call $x$ as a \emph{projective fixed  point} of $A$ corresponding to $[\lambda]$. If $A$ is regular, it fixes exactly $n+1$ points on $\P(\V)$ and thus, it has $n+1$ projective fixed points. 
\end{definition} 

\begin{remark} 
We emphasize here that the projective fixed points of $A$ are not the same as the projective points of $A$. The notion of the projective points of $A$ is elaborated in \secref{pp}. 
\end{remark}

\begin{lemma}\label{emb}
The group $\Sp(n,1)$ can be embedded in the group ${\rm GL}(2n+2, \C)$.
\end{lemma}
\begin{proof}
Write $\H=\C\oplus{\bf j}  \C$. For $A\in \Sp(n,1)$, express $A=A_1+{\bf j}A_2$,
{where} $ A_1, A_2\in M_{n+1}(\C)$. The correspondence $A \mapsto A_{\C}$, where 
\begin{equation}\label{crep} A_{\C}= \begin{bmatrix}  A_1 &  -\overline{A_2} \\
                          {A_2}  & \overline{A_1}    \end{bmatrix},                   
\end{equation}
embeds $\Sp(n,1)$ into ${\rm GL}(2n+2, \C)$. 
\end{proof}
The following lemma is a special case of \cite[Theorem 3.1]{gop}. 
\begin{lemma}\label{rt}
Let $A$ be an element in $\Sp(n,1)$. Let $A_{\C}$ be the
corresponding element in ${\rm GL}(2n+2, \C)$. The characteristic polynomial
of $A_{\C}$ is of the form
\begin{equation*}\chi_A(x)=\sum_{j=0}^{2n+2} a_j x^{2(n+1)-j},\end{equation*}
where $a_0=1=a_{2n+2}$ and for $1 \leq j \leq n+1$, $a_j=a_{2(n+1)-j}$.  Write     \hbox{$\chi_A(x)=x^{n+1} g(x+x^{-1})$}. Let $\Delta$ be the negative of the discriminant of the polynomial $g_A(t)=g(x+x^{-1})$. Then $A$ is  regular loxodromic if and only if, $\Delta >0$ and $  \sum_{j=0}^n  a_j \neq -\frac{1}{2} a_{n+1} \neq  \sum_{j=0}^n (-1)^{n+1-j}  a_j$.
The conjugacy class of  $A$ is determined by the real numbers $a_j$, $1 \leq j \leq n+1$. 
\end{lemma}
\begin{proof} Note that $g(x+x^{-1})=\sum_{j=0}^n (x^{n+1-j}+ x^{-(n+1-j)}) + a_{n+1}$. 
It is proved in \cite[Theorem 3.1]{gop} that $A$ is regular hyperbolic if and only if $\Delta>0$. Now, $A$ has no eigenvalue $\pm 1$ if and only if $g(\pm 2) \neq 0$, i.e.
$a_{n+1} + 2 \sum_{j=0}^n  a_j \neq 0 \neq a_{n+1} + 2 \sum_{j=0}^n (-1)^{n+1-j} a_j$.
\end{proof} 
\begin{definition}
Let $A$ be a regular loxodromic element in $\Sp(n,1)$. The $(n+1)$-tuples of real numbers $(a_1, \ldots, a_{n+1})$ as in  \lemref{rt} will be called the  \emph{real trace} of $A$ and we shall denote it by $tr_{\R}(A)$.
\end{definition}
\subsection{Useful results} 
We shall use the following result by Cao \cite{cao} that determines quadruples of points on  $\partial \h^n$. We refer to \cite{cao} or \cite{ak} for the basic notions of angular invariants. For the notations used in the following statement, see \cite[Section 2]{gk2}. 
\begin{theorem}\label{cao} {\rm  \cite{cao}}
Let $Z=(z_1, z_2, z_3, z_4)$ and $W=(w_1, w_2, w_3, w_4)$ be two quadruples of pairwise distinct points in $\partial \h^n$. Then there exists an isometry $h \in \Sp(n,1)$ such that $h(z_i)=w_i$, $i=1,2,3,4$, if and only if the following conditions hold:
\begin{enumerate}
\item For $j=1,2,3$, $\X_j(z_1, z_2, z_3, z_4)$ and $~\X_j(w_1, w_2, w_3, w_4)$ belong to the same \hbox{similarity}  class.

\item $\A(z_1, z_2, z_3)=\A(w_1, w_2, w_3)$, ~$\A(z_1, z_2, z_4)=\A(w_1, w_2, w_4)$, ~ $\A(z_2, z_3, z_4)=\A(w_2, w_3, w_4)$.
\end{enumerate}
\end{theorem}
Cao also proved that, for $n \geq 3$,  the moduli space of $\Sp(n,1)$-congruence classes of points is homeomorphic to a semi-algebraic subspace of $\C^3 \times \R \times \R$ defined by these invariants.

\medskip In the complex hyperbolic set up, the moduli of ordered quadruples of points was obtained by Cunha and Gusevskii. We recall their result.
\begin{theorem}\cite{cugu2}\label{cg22}
Let $Z=(z_1, z_2, z_3, z_4)$ and $W=(w_1, w_2, w_3, w_4)$ be two quadruple of pairwise distinct points in $\partial \ch^n$. Then there exists an isometry $h \in \SU(n,1)$ such that $h(z_i)=w_i$, $i=1,2,3,4$, if and only if the following conditions hold:
\begin{enumerate}
\item $\A(z_1, z_2, z_3)=\A(w_1, w_2, w_3)$.
\item $\X(z_1, z_2, z_3, z_4)=\X(w_1, w_2, w_3, w_4)$, $\X(z_1, z_4, z_2, z_3)=\X(w_1, w_4, w_2, w_3)$. 
\end{enumerate}
Further, these invariants $(\X(z_1, z_2, z_3, z_4), \X(z_1, z_4, z_2, z_3), \A(z_1, z_2, z_3))$ form a semi-algebraic subset of $\C^2\setminus \{\{0\} \times \R\}$ which is homeomorphic to the moduli space. 
\end{theorem} 

%\medskip Note that there is a natural action of the symmetric group $S_4$ on $\M(2)$, coming from the $S_4$ action on an ordered tuple:  
%$$g. [(p_1, p_2, p_3, p_4)]=[(p_{g(1)}, p_{g(2)}, p_{g(3)}, p_{g(4)})].$$
%The orbit space of $\M(2)$ under this action will be denoted by $\mathcal M(2)$. 

\section{ Projective Points}\label{pp} 

\subsection{Projective points} We recall the concept of projective points from \cite{gk1}. Let $T$ be an invertible matrix over $\H$. Let $\lambda \in \H\setminus\R$ be a chosen eigenvalue of $T$ in the similarity class $[\lambda]$.  Identify the $[\lambda]$-eigenspace with $\H$.  Consider the $\lambda$-\emph{eigenset}: $S_{\lambda}=\{ x \in V \ | \ Tx =x \lambda \}$. Note that this set is $x Z(\lambda)$ that is a copy of $\C$ in $\H$. Now, identify $\H$ with $\C^2$. Two non-zero quaternions $q_1$ and $q_2$ are equivalent if  $q_2=q_1 c$, $c \in \C \setminus 0$. This equivalence relation projects $\H$ to the one dimensional complex projective space $\C \P^1$, the $[\lambda]$-eigensphere. 
Since $[\lambda]$ is a conjugacy invariant of $T$, so also the $[\lambda]$-eigensphere $\C \P^1$.

Let $v$ be the projection of the $[\lambda]$-eigenspace. Then for each point on $\C \P^1$, there is a choice of the lift $\v$ of $v$ that spans a complex line in $\v\H$. This choice of $\v$ corresponds to the eigenset of the eigenvalue $\lambda$ of $\v$, and the corresponding point on the eigensphere $\C\P^1$ is called a \emph{projective point} of $[\lambda]$. 

\subsection{Projective points and loxodromic elements} Now suppose $A$ is a regular loxodromic element in $\Sp(n,1)$.
If $a_A$ and $r_A$ are the fixed-points of $A$, then we can determine projective point corresponding to $r_A$, if we know the projective point corresponding to $a_A$ on $\C \P^1$. So we require a single projective point corresponding to pair $(a_A, r_A)$ on $\C \P^1$. Here we have used the fact that $Z(\lambda)=Z(\bar \lambda^{-1})$. Similarly, the projective points of $\x_{1,A}, \ldots, \x_{n-1,A}$ correspond to the centralizer  $Z(\mu_1), \ldots,  Z(\mu_{n-1})$ respectively.

\medskip   The following classification of loxodromic elements in $\Sp(n,1)$ follows from \cite[Section 4.1]{gk2}.

\begin{lemma}\label{lox}
Let $A$ and $A'$ be regular loxodromic elements in $\Sp(n,1)$. Then $A=A'$ if and only if they have the same projective fixed points, the same real trace,  and the same projective points.
\end{lemma}

The above lemma may be interpreted as follows. 
Suppose $\mathcal C$ be the $\Sp(n,1)$ conjugacy classes of regular loxodromic elements. It follows from \lemref{hycl} that the real traces classify a point on $\mathcal C$, and up to conjugacy we can assume that elements of $\mathcal C$ have the same projective fixed points. Let $\mathcal T$ be the set of real traces $(a_1, \ldots, a_{n}) \in \R^{n}$ given by  $\Delta^{-1}(0, \infty)$, where $\Delta: \mathcal C \to (0, \infty)$ is the discriminant function in \lemref{rt}. There is a natural projection map $p: \mathcal C \to \mathcal T$. However, $p^{-1}(t)$ is not unique. The map $p$ has  fiber  $(\C \P^1)^n=\C \P^1 \times \cdots \C \P^1$. A point on this $(\C\P^1)^n$ determines a loxodromic element uniquely up to relabelling of fixed points.

\medskip In the case of $\SU(n,1)$ an easier version of the above lemma holds true. 
\begin{lemma} \label{lox1}
Let $A$ and $A'$ be regular loxodromic elements in $\SU(n,1)$. Then $A=A'$ if and only if they have the same projective fixed points and the same characteristic polynomial, where having the same characteristic polynomial is equivalent to the condition of having the same eigenvalues. 
\end{lemma} 

\section{Weakly Non-singular Pairs}\label{ns}
In this section, we mostly work with the group $\Sp(n,1)$. However, the arguments restrict over $\SU(n,1)$ with slight modifications, and hence omitted. 
\subsection{Gram matrix associated to a pair} \label{nor}
Let $(A, B)$ be a weakly non-singular pair in $\Sp(n,1)$. Condition (3) in \defref{nsl}  implies that we may assume,   by re-arranging the indices if necessary, that 
$$\langle \x_{k,A},\a_{B}\rangle\neq 0, ~ \langle \x_{k,B},\a_A\rangle \neq 0,   ~ \langle \r_A,\x_{k,B}\rangle \neq 0, ~\langle \r_B,\x_{k,A}\rangle \neq 0, \text{ for }1\leq k \leq n-2.$$

We normalize  the eigenvectors such that for $1\leq k \leq n-2$, 
\begin{equation} \label{nor} 
\langle \a_A,\r_A \rangle=\langle \a_A,\a_B\rangle=\langle \a_A,\r_B \rangle=\langle \a_A,\x_{k,B}\rangle =\langle \a_B,\x_{k,A}\rangle=1,  | \langle \a_B, \r_A \rangle|=1,\end{equation} 
and $\langle \r_A,\x_{k,B}\rangle \neq 0 \neq \langle \r_B, \x_{k,A}\rangle$. 

\medskip For simplicity of notations, we write 
\begin{itemize}
	\item [$\bullet$]$p_1=a_A, ~p_2=r_A, ~p_3=a_B, ~ p_4=r_B,$
	\item [$\bullet$]$ \hbox{for, }5 \leq j\leq  n+2,~  p_j=x_{j-4, A},$
	\item [$\bullet$]$\hbox{for, }n+3 \leq j \leq 2n,~  p_j=x_{j-(n+2), B}.$
	\item [$\bullet$]$p_{2n+1}=x_{n-1, A}$, ~ $p_{2n+2}=x_{n-1, B}$.
\end{itemize}

\medskip Since the eigenvectors of $A \in \Sp(n,1)$ form  an  orthonormal basis for $\H^{n,1}$, it follows that if $C(p_i)=p_i'$ for $1 \leq i \leq 2n$, then $C(p_j)=p_j'$ for $j=2n+1, 2n+2$. For this reason, we shall 
 associate to $(A, B)$ the Gram matrix $(g_{ij}),~g_{ij}=\langle \p_i, \p_j \rangle$,  of the ordered $2n$-tuple $p=(p_1, p_2, \ldots, p_{2n})$. 
In view of the normalized eigenvectors, the Gram matrix has the form $G(p)=(g_{ij})$, where
\begin{enumerate}
\item  $g_{11}=g_{22}=g_{33}=g_{44}=0$; $~~g_{12}=g_{13}=g_{14}=1= \lvert g_{23} \rvert$. 
\medskip \item  For $5 \leq j \leq n+2$, $g_{1j}=0$, $g_{2j}=0$; and,  for $n+3 \leq k \leq 2n$, $g_{1k}=1$, $g_{2k} \neq 0$. 

\medskip \item For $5 \leq j \leq n+2$, $g_{3j}=1$, $g_{4j}\neq 0$; and,  for $n+3 \leq k \leq 2n$, $g_{3k}=0$, $g_{4k} = 0$. 
\medskip \item For $5 \leq j, k \leq n+2$, $j<k$, $g_{jk}=0$; and, for $n+3 \leq k,j \leq 2n$, $g_{jk}=0$,$j<k$, $g_{jk} = 0$.
\end{enumerate}
We call $G$ a \emph{normalized Gram matrix} associated to $(A, B)$. 

\begin{lemma}\label{4.1o}  Suppose that the Gram matrix $G(\p)$ is a normalized Gram matrix for $p$ with respect to the lift $\p=(\p_1,\p_2,\ldots,\p_{2n})$. Let $G({\p}')$  be the  normalized Gram matrix with respect to the lift $\p'=(\p_1{\lambda_1},\ldots,\p_{2n}{\lambda_{2n}})$ of $p$. Then $\lambda_1=\lambda_2=\ldots=\lambda_{2n}$ and ${\lambda_1} \in \Sp(1)$.\end{lemma} 
\begin{proof}
We have ${\langle \p_1\lambda_1,\p_k{\lambda_k}\rangle}=1$, thus $\overline{\lambda_k}{\lambda_1} = 1$, for $k= 2,3,
4$ because ${\langle \p_1,\p_k\rangle}=1$. Now from $\lvert{{\langle \p_2{\lambda_2},\p_3{\lambda_3}\rangle}}\rvert = 1$, we have
$\lvert {\overline{\lambda_3}}\rvert{ \lvert \lambda_2 \rvert }=1 $ as $\lvert \langle \p_2,\p_3\rangle \rvert =1$. Thus we have $\vert\lambda_1 \rvert = 1 $ so $\lambda_1 \in \Sp(1)$. Therefore by $\overline{\lambda_k}{\lambda_1} = 1$ for $k= 2,3,4$ we have $\lambda_1=\lambda_2=\lambda _3=\lambda_4$ and $ {\lambda_1}$ $ \in \Sp(1)$. 

By ${\langle \p_3{\lambda_3},\p_j{\lambda_j}\rangle}= 1$, for $j=5,6,\ldots,{n+2}$ we have $\overline{\lambda_j}{\lambda_3} = 1$. Thus $\lambda_3= \lambda_j$, for $j=5,6,\ldots,{n+2}$ satisfies from $\lvert{{\lambda_3\lvert}}=1$.
Also from the relations ${\langle \p_1{\lambda_1},\p_k{\lambda_k}\rangle}= 1$, for $k={n+3},n+4,\ldots,2n$ we can see that $\overline{\lambda_k}{\lambda_1} = 1$, for $k={n+3},n+4,\ldots,2n$. Now $\lvert{{\lambda_1\lvert}}=1$ gives $\lambda_1= \lambda_k$ for $k={n+3},n+4,\ldots,2n$.  So we have $\lambda_1=\lambda_2=\ldots=\lambda_{2n}$ and ${\lambda_1} \in \Sp(1)$.
\end{proof}
The Gram matrix $G(\p)$ is well-defined up to a scalar action of $\Sp(1)$. We denote the $\Sp(1)$ orbit of entries of $G(\p)$ as $O_{G(p)}$. The following theorem follows using similar arguments as in the proof of \cite[Lemma 8.9]{gk2}. 
\begin{lemma}\label{leg}
Let $(A, B)$ and $(A', B')$ be two weakly non-singular pairs of loxodromic elements in $\Sp(n,1)$.  Let $p=(p_1, \ldots, p_{2n})$ and $p'=(p_1', \ldots, p_{2n}')$ be the associated tuples to the pairs respectively.  Then there exists $C \in \Sp(n,1)$ such that $C(p_i)=p_i'$, $i=1,\ldots, 2n$,  if and only if $O_{{G(p)}}=O_{{G(p')}}$. 
\end{lemma}

\begin{remark} We note further that, if we keep the lift of a chosen point $p_j$ from the same hyperplane, e.g. if we always take $\p_j$ to be standard, then it follows from \lemref{4.1o} that there is a unique normalized Gram matrix associated to the tuple $p$. \end{remark}

\subsection{Conjugacy invariants} \label{ci} We consider the following invariants associated to the tuple $p$.  

\begin{enumerate}
\item Angular invariant: $\A(p_1, p_2, p_3)$. 

\medskip 
\item  Usual Cross-ratios: $ \X_1(A, B)=\X(p_1, p_2, p_3, p_4), ~ \X_2(A, B)=\X(p_1, p_3, p_2, p_4)$. 

\medskip \item {Generalized Cross-ratios}:\\

  For $n+3 \leq k \leq 2n$, $\X_{2k}(A, B)=\X(p_1, p_2, p_3, p_j)$. \\

For $5 \leq j \leq n+2$, $\X_{4j}(A, B)=\X(p_3, p_4, p_1, p_j)$.\\

For $5 \leq j \leq n+2$,  $n+3\leq k \leq 2n$, $\X_{jk}(A, B)=\X(p_3,p_k,p_2,p_j)$.

\medskip Note that we have denoted $\X_{2k}(A, B)$ by $\alpha_k(A, B)$ and $\X_{4j}(A, B)$ by $\beta_k(A, B)$ in \secref{intro}. 

\medskip \item {Goldman's eta-invariants}: 
\\For $5 \leq j \leq n+2$, $\eta_{j}(A, B)=\eta(p_3, p_4; p_j)= {\langle {\bf p}_3, {\bf p}_j \rangle \langle {\bf p}_3, \bf p_4 \rangle}^{-1} { \langle {\bf p}_j, {\bf p}_4\rangle \langle   {\bf p}_j, {\bf p}_j \rangle^{-1}}$.\\
 For $n+3 \leq k \leq 2n$, $\eta_k(A, B)=\eta(p_1, p_2; p_k) = {\langle {\bf p}_1, {\bf p}_k \rangle \langle {\bf p}_1, \bf p_2 \rangle}^{-1} { \langle {\bf p}_k, {\bf p}_2\rangle \langle   {\bf p}_k, {\bf p}_k \rangle^{-1}}$.
\end{enumerate}

\medskip We note that using our notation earlier, $\X_{2j}(A, B)=\alpha_j(A, B)$, and $\X_{4k}(A, B)=\beta_k(A, B)$. However, we slightly change the notation here in order to have uniformity in the symbols. 

\medskip 
\begin{lemma} \label{leg2}
Let $(A, B)$ be a weakly non-singular  pair in $\Sp(n,1)$. Suppose that the Gram matrix $G(\p)=(g_{ij})$ is a normalized Gram matrix associated to $(A, B)$ with respect to the lift $\p=(\p_1,\p_2,\ldots,\p_{2n})$. Then  the Gram matrix is determined by the invariants listed above.   
\end{lemma}
\begin{proof}The proof is obtained by computing the invariants in view of the normalized Gram matrix and we have
$$\A=\arg(-g_{23}), \hbox{i.e. } g_{23}=-e^{i \A}$$
$$\X_1={\overline g_{23}}^{-1} {\overline g_{24}}, ~\X_2=g_{23}^{-1} {\overline g_{34}};$$
$$\X_{2k}={\overline g_{23}}^{-1} {\overline g_{2k}},~ \X_{4j}=\overline{g}_{4j};$$
$$\X_{jk}=g_{23} g_{2k}^{-1} g_{jk}, ~5 \leq j \leq n+2, ~ n+3\leq k \leq 2n;$$
$$\eta_j=g_{34}^{-1} \overline g_{4j} g_{jj}^{-1}, ~ \eta_k = {\overline g_{2k}} g_{kk}^{-1} .$$
This clearly shows the result. 
\end{proof}

\subsection{Classification of weakly non-singular pairs} 
\begin{theorem}\label{thmp2}
Let $(A, B)$ be a weakly non-singular pair of loxodromic elements in $\Sp(n,1)$. 
 Then $(A, B)$ is determined uniquely up to conjugacy in $\Sp(n,1)$ by the real traces, the angular invariant $\A(a_A, r_A, a_B)$, the $\Sp(1)$  conjugation orbit of the (unordered) tuple of the above conjugacy invariants (2)--(4),  and the projective points. \end{theorem}
\begin{proof} Let $(A, B)$ and $(A', B')$ be loxodromic elements in $\Sp(n,1)$. Suppose  $p=(p_1, \ldots, p_{2n})$ and $p'=(p_1', \ldots, p_{2n}')$ are the associated tuples to the pairs respectively. 
Assume that $\A(p_1, p_2, p_3)=\A({p_1}', {p_2}', {p_3}')$, and the $\Sp(1)$  conjugation orbit of the (unordered) tuple of the above conjugacy invariants (2)--(4) with respect to $(A,B)$ and  $(A',B')$ are equal. So there exist $\mu \in \Sp(1)$ such that 
$$\mu \X_1(A, B)\bar {\mu}=\X_1(A', B'),~ \mu \X_2(A, B)\bar {\mu}=\X_2(A', B'), ~ \mu \X_{2k}(A, B)\bar {\mu}=\X_{2k}(A', B'),$$ 
$$\mu \X_{4j}(A, B)\bar {\mu}=\X_{4j}(A', B'),~\mu \X_{jk}(A, B)\bar {\mu}=\X_{jk}(A', B'),~\mu \eta_j(A, B)\bar {\mu}=\eta_j(A', B'),$$
$$\mu \eta_k(A, B)\bar {\mu}=\eta_k(A', B').$$
	
By Lemma \ref{leg2}, we have  $DG(\p)D^{-1}= G(\p')$, where $D=diag(\mu,\mu,\ldots,\mu)$. That is, $O_{{G(\p)}}=O_{{G(\p')}}$. Then by \lemref{leg}, there exists $C \in \Sp(n,1)$ such that $C(p_i)=p_i'$, for $1 \leq i \leq 2n+2$. In particular, $CAC^{-1}$ and $A'$ have the same projective fixed points. Since they have the same real traces, they belong to the same conjugacy class.  By \lemref{lox}, $CAC^{-1}=A'$ if and only if they have the same projective points. Similarly, $CBC^{-1}=B'$. 
\end{proof} 

\begin{remark}
Let $\mathcal I$ denote the tuple of real numbers given by the above invariants, and let $\mathcal T$ denote the set of real traces of regular  loxodromics.  Let $\mathcal W$ denote the set of weakly non-singular representations in $\fX({\mathrm F}_2, \Sp(n,1))$. Clearly by \lemref{leg2} there is a well-defined map 
$p: \mathcal{W} \to \mathcal T \times \mathcal T \times \mathcal I$. However, given a point $t$ in the image $p(\mathcal{W})$, $p^{-1}(t)$ is not a unique point, but a product of $2n$ copies of $\C\P^1$ corresponding to the projective points. 
\end{remark}
\subsubsection{Proof of \thmref{thm2}} This is a restatement of the above theorem where $\rho(x)=A$, $\rho(y)=B$. 

\subsubsection{Proof of \thmref{cor3}} Follows from the above by restricting everything over $\C$. 

\section{The Non-Singular Pairs}\label{mnth}  
\begin{lemma}\label{csprp1}
 Let $A,~B$ be loxodromic elements in $\SU(n,1)$ such that  $(A, B)$ is non-singular. Denote $\A(A, B)=\A(a_A, r_A, a_B)$. Let $(A', B')$ be a non-singular and loxodromic pair such that the following holds: 

\medskip 
(i) For $k=1,2$, $\X_{k}(A,B)=\X_{k}(A',B')$, $\A(A, B)=\A(A',B')$.

(ii) For $1 \leq j \leq n-2$, $\alpha_j(A', B')=\alpha_j(A, B)$ and  $\beta_j(A', B')=\beta_j(A, B)$. 

\medskip   Then
there exists an element $C$ in $\SU(n,1)$ such that $C(a_{A})=a_{A'},~C(r_{A})=r_{A'}$, $C(x_{k, A})=x_{k, A'}$,  and, 
$C(a_{B})=a_{B'},~C(r_{B})=r_{B'}$, $C(x_{k, B})=x_{k, B'}$.   
\end{lemma}
\begin{proof}  We shall follow similar arguments as in the proof of \cite[Lemma 5.1]{gp2}.

\medskip Since $\X_{k}(A,B)=\X_{k}(A',B'), \A(A, B)=\A(A',B') ~k=1,2$, by \thmref{cg22} it follows that there exist $C\in \SU(n,1)$ such that 
$C(a_{A})=a_{A'},~C(r_{A})=r_{A'},~C(a_B)=a_{B'}$ and  $C(r_{B})=r_{B'}$. Let $1 \leq k \leq n-2$. 
Since $\alpha_k(A, B)=\alpha_k(A', B')$, hence 
$${\langle \xb, \ra \rangle \langle \xb, \aa \rangle}^{-1} { \langle \ab, \aa\rangle} {\langle \ab, \ra \rangle^{-1}}={\langle \yxb, \yra \rangle \langle \yxb, \yaa \rangle}^{-1} { \langle \yab, \yaa \rangle} {\langle \yab, \yra \rangle^{-1}}$$
Let $${\langle C^{-1}(\yxb), \ra \rangle}^{-1} \langle \xb, \ra \rangle =  {\langle C^{-1}(\yxb), \aa \rangle^{-1}} \langle \xb, \aa \rangle = \lambda$$

This implies
\begin{equation}\label{e11}  \langle \xb-C^{-1}(\yxb){\lambda}, \ra \rangle=0; \end{equation} 
\begin{equation} \label{e22}  \langle \xb- C^{-1}(\yxb){\lambda}, \aa \rangle=0.\end{equation}
On the other hand, note that
\begin{equation}\label{e33}  \langle \xb-C^{-1}(\yxb){\lambda}, \rb \rangle =\langle \xb, \rb\rangle- \langle C^{-1}(\yxb), \rb \rangle\lambda=0-  \langle \yxb, \yrb \rangle \lambda=0. \end{equation} 

Similarly,
\begin{equation} \label{e44} \langle  \xb- C^{-1} (\yxb){\lambda}, \ab \rangle=0.\end{equation}
Let $L_A$  and $L_B$ denote the two-dimensional time-like subspaces of $\C^{n,1}$ with $\{\aa, \ra\}$ and $\{\ab, \rb\}$ are the respective bases of $L_A$ and $L_B$, that represents the complex lines.  Thus it follows from \eqnref{e11} - \eqnref{e44} that $\v= \xb- C^{-1} (\yxb){\lambda}$ is orthogonal to  both $L_A$ and $L_B$.  We must have $\langle v, v \rangle >0$. Thus $v$ is polar to the $(n-1)$ dimensional totally geodesic complex subspace  that is represented by ${\rm V}=v^{\perp}$. Since $\C^{n,1}={\rm V} \oplus  \C v $, hence $L_A$ and $L_B$ must be subsets in ${\rm V}$. Thus, the fixed points of $A$ and $B$ belong to the boundary of the totally geodesic subspace $\P({\rm V})$. This is a contradiction to the non-singularity of $(A, B)$. Hence we must have $v=0$, that is $C(\xb)= \yxb \lambda$. Thus,  $C(x_{k,B})=x_{k,B'}$. Consequently, $C(x_{n-1, B})=x_{n-1, B'}$. 

Similarly $\beta_j(A, B)=\beta_j(A', B')$ implies $C(x_{j, A})=x_{j, A'}$ for $1 \leq j \leq n-1$.  This proves the lemma. 
\end{proof}

\subsection{Proof of \thmref{cor2}}
If $(A, B)$ and $(A', B')$ are conjugate, then it is clear that they have the same invariants. 

Conversely, suppose $(A, B)$ and $(A', B')$ are non-singular pairs of loxodromics such that $\alpha_k(A, B)=\alpha_k(A', B')$, $\beta_k(A, B)=\beta_k(A', B')$, $1 \leq k \leq n-2$,  $\X_i(A, B)=\X_i(A', B')$, $i=1, 2$, $\A(A, B)=\A(A', B')$. By \lemref{csprp1}, it follows that there exist $C\in \SU(n,1)$ such that 
$C(a_{A})=a_{A'},~C(r_{A})=r_{A'}, ~C(x_{k,A})=x_{k, A'}$ and $C(a_{B})=a_{B'},~C(r_{B})=r_{B'}, ~C(x_{k, B})=x_{k, B'}$, $1 \leq k \leq n-1$. 
Therefore $A'$, resp. $B'$,  and $CAC^{-1}$, resp. $CBC^{-1}$,  have the same fixed points. Since they also have the same family of traces,  $CAC^{-1}=A'$. Similarly, $CBC^{-1}=B'$. 
 This completes the proof.

\subsection{Proof of \thmref{qc2}} The following lemma follows by mimicking the proof of \thmref{csprp1}, the only difference is that instead of \thmref{cg22}, one has to apply \thmref{cao} in the proof.  

\begin{lemma}
Let $A,~B$ be loxodromic elements in $\Sp(n,1)$ such that  $(A, B)$ is non-singular. Suppose the lifts of the attracting fixed points of a loxodromic element are always assumed to be standard. 
  Let $(A', B')$ be a non-singular pair such that the following holds: 

\medskip 
(i) For $k=1,2,3$, $\X_{k}(A,B)=\X_{k}(A',B')$, $\A_k(A, B)=\A_k(A',B')$.

(ii) For $1 \leq j \leq n-2$, $\alpha_j(A', B')=\alpha_j(A, B)$ and  $\beta_j(A', B')=\beta_j(A, B)$. 

\medskip   Then
there exists an element $C$ in $\Sp(n,1)$ such that $C(a_{A})=a_{A'},~C(r_{A})=r_{A'}$, $C(x_{k, A})=x_{k, A'}$,  and, 
$C(a_{B})=a_{B'},~C(r_{B})=r_{B'}$, $C(x_{k, B})=x_{k, B'}$.   
\end{lemma}
Now, \thmref{qc2} follows using same arguments as above or in the proof of \thmref{thm2}. 
\section{The Twist-Bend Parameters and Surface group Representations} \label{twb} 
\subsection{The Twist-Bend Parameters } Suppose that $\langle A, B \rangle$ is a non-singular $(0,3)$ group in ${\rm Sp}(3,1)$, i.e. $A$, $B$ and $B^{-1}A^{-1}$ are loxodromics and $\langle A, B \rangle$ is free. We shall also assume that  $(A, B)$ is non-singular. We want to attach two such non-singular subgroups to get a group that is freely generated by three generators. Now two cases are possible.  The first case corresponds to the case when two different three-holed spheres (or pair of pants) are attached along their boundary components. This gives a $(0, 4)$ group generated by three elements. The second case corresponds to the case when two of the boundary components of the same three-holed sphere is glued. In this case gluing two $(0,3)$ groups gives an $(1,1)$ group that is a group generated by two loxodromic elements and their commutator. This process is called `closing a handle'. To get more details of these terminologies and the gluing process, we refer to \cite{pp}. 

 Let $\langle A, B \rangle$ and $\langle C, D \rangle$ be two non-singular $(0,3)$ groups in ${\rm Sp}(3,1)$ such that the boundary components associated to $A$ and $D$ are compatible. Here compatibility means $A=D^{-1}$. A three dimensional \emph{quaternionic hyperbolic twist bend} corresponds to an element $K$ in ${\rm Sp}(3,1)$ that commutes with $A$ and conjugates $\langle C, D \rangle$, see \cite[Section 8.1]{pp}. We assume that up to conjugacy, $A$ fixes $0$, $\infty$, and it is of the form $E(r, \theta, \phi_1, \phi_2)$. Since $K$ commutes with $A$, it is also of the form $K=E(t, \psi, \xi_1, \xi_2)$, see \cite{kgz}. Thus $K$ is either a boundary elliptic or, a hyperbolic element. 

It follows that there is a total of ten real parameters associated to $K$, the real trace $(t, \psi, \xi_1, \xi_2)$, along with six real parameters associated to the projective points. If $t=1$, then $K$ is a boundary elliptic and the eigenvalue $[e^{i \psi}]$ has multiplicity $2$.  The projective points for these eigenvalues can be defined as before.  There are exactly one negative-type and two positive-type eigenvalues of $K$. Since $K$ commutes with $A$, the projective points of $K$ is determined by the projective points of $A$. Hence,  there are three projective points of $K$ to determine it. Consequently, we shall have 10 real parameters associated to a twist-bend $K$. We denote these parameters by $\kappa=(t, \psi, \xi_1, \xi_2, k_1, k_2, k_3)$, where $k_1=p_1(K)$, $k_2=p_2(K)$, $k_3=p_3(K)$ are the projective points of the  similarity classes of eigenvalues of $K$.

The parameters  $\kappa=(t, \psi, \xi_1, \xi_2, k_1, k_2, k_3)$ obtained this way, is called the \emph{twist-bend parameters}. Note that the twist-bend is a relative invariant as it  always has to be chosen with respect to some fixed group $\langle A, B, C \rangle$ that one has to specify before applying the twist-bend.  When we write $A=Q E(r, \theta, \phi_1, \phi_2) Q^{-1}$, if the matrix $K=QE(t, \psi, \xi_1, \xi_2)Q^{-1}$, then we say that the twist-bend parameters $\kappa$ is \emph{oriented consistently} with $A$.  

To obtain conjugacy-invariants to quantify the twist-bend parameters, we define the following numerical objects corresponding to $\kappa$:
$$\tilde \X_1(\kappa)=\X(a_A, r_A, a_B, K(r_C)),~ \tilde \X_2(\kappa)=\X(a_A, K(r_C), a_B, r_A),~\tilde \X_3(\kappa)=\X(r_A, K(r_C), a_B, a_A);$$
$$\tilde \A_1(\kappa)=\A(a_A, r_A, K(r_C)), ~~\tilde \A_3(\kappa)=\A(r_A, K(r_C), a_B).$$
%$$\tilde \beta_1(\kappa )=[K(\r_C), \a_B, \x_A, \a_A], ~ \tilde \beta_2(\kappa)=[K(\r_C), \a_B, \y_A, \a_A].%$$
\begin{lemma}\label{tw1} 

Let $A$, $B$, $C$ be loxodromic transformations of $\h^3$ such that $\langle A, B \rangle$ and $\langle A^{-1}, C \rangle$ are non-singular $(0, 3)$ subgroups of $\Sp(3,1)$.  We further assume that $a_B$, $r_C$ do not lie on a proper totally geodesic subspace joining $a_A$ and $r_A$. Let $K=E_K(t, \psi, \xi_1, \xi_2, k_1, k_2, k_3)$ and $K'=E_{K'}(t', \psi', \xi'_1, \xi'_2, k_1', k_2', k_3')$ represent twist-bend parameters that are oriented consistently with $A$. If
$$[\tilde \X_1(\kappa)]=[\tilde \X_1(\kappa')], ~~[\tilde \X_2(\kappa)]=[\tilde \X_2(\kappa')], ~~[\tilde \X_3(\kappa)]=[\tilde \X_3(\kappa')]; $$
$$\tilde \A_1(\kappa)=\tilde \A_1(\kappa'), ~\tilde \A_3(\kappa)=\tilde \A_3(\kappa'); $$
and $k_1=k_1'$, $k_2=k_2'$, $k_3=k_3'$, 
then $K=K'$.
\end{lemma}
\begin{proof}
Without loss of generality we assume $a_A=o$, $r_A=\infty$.
In view of the conditions
$$[\tilde \X_1(\kappa)]=[\tilde \X_1(\kappa')], ~~[\tilde \X_2(\kappa)]=[\tilde \X_2(\kappa')],~~[\tilde \X_3(\kappa)]=[\tilde \X_3(\kappa')], \hbox{ and }$$
$$\tilde \A_1(\kappa)=\tilde \A_1(\kappa'), ~\tilde \A_3(\kappa)=\tilde \A_3(\kappa'), $$ and noting that  $\tilde \A_2(\kappa)$ and $\tilde \A_2(\kappa')$ are trivially equal,  following similar arguments as in the proof of \cite[Theorem 5.2]{cao}, we have $f$  in $\Sp(3,1)$ such that $f(a_A)=a_A$, $f(r_A)=r_A$, $f(a_B)=a_B$ and $f(E_K(r_C))=E_{K'}(r_C)$. Since $f$ fixes three points on the boundary, it must be of the form  
$$f=\begin{bmatrix} \mu_o & 0 & 0 & 0  \\ 0 & \mu_o & 0 & 0 \\0 & 0 & \mu_1 & 0 \\ 0 & 0 & 0 & \mu_2 \end{bmatrix}.$$
The boundary fixed point set of such a transformation always bounds a proper totally geodesic subspace of $\h^3$.  Since $a_B$ does not lie on a proper totally geodesic subspace joining $a_A$ and $r_A$, we must have $f=\pm I$. Thus, it follows that $E_{K}(r_C)= E_{K'}(r_C)$. Now by using the fact that $E_K E_{K'}^{-1}$ has the three fixed points $a_A= o, r_A=\infty $ and $r_C$ together with the condition that $r_C$ does not lie on a totally geodesic subspace joining $a_A$ and $r_A$, we have $E_K =E_{K'}$. 

 Hence, $K$ and $K'$ are conjugate with the same attracting and the same repelling points. So, by \lemref{lox}, $K=K'$ if and only if they have the same projective points and  the same fixed points. This completes the proof.
\end{proof}
\subsection{Proof of \thmref{mth2}} After we have \thmref{qc2} and \lemref{tw1}, the proof of \thmref{mth2} follows by mimicking the arguments in \cite{pp} or \cite{gk1}. We sketch it here. 

Let $\Sigma_g \setminus \mathcal C$ be the complement of the curve system $\mathcal C$ in $\Sigma_g$. This is a disjoint union of $2g-2$ three holed spheres. Each of the three-holed sphere corresponds to a  non-singular $(0, 3)$ subgroup of ${\Sp}(3,1)$. By \corref{qcc2}, a $(0, 3)$ subgroup $\langle A, B \rangle$ is determined up to conjugacy by the $36$ real parameters.  While attaching two three-holed spheres, we attach two $(0, 3)$ groups subject to the compatibility condition that a peripheral element in one group is conjugate to the inverse of a peripheral element in the other group. This gives a $(0, 4)$ group that can be seen to be determined by $72$ real parameters. Proceeding this way, attaching  $2g-2$ of the above $(0, 3)$ groups, we get a surface with $2g$ handles, and it is determined by  $36(2g-2)=72g-72$ real parameters obtained from the attaching process. The handles correspond to the $g$ curves that in turn correspond to the two boundary components of the  three-holed spheres. Now,  there are $g$ quaternionic constraints that are  imposed to close these handles: one of the peripheral elements of each of these $(0,3)$ groups  must be conjugate to the inverse of the other peripheral element. Note that, corresponding to each peripheral element there are $10$ natural real parameters: the real trace and two projective points. So, the number of real parameters reduces to $72g-72-10g=62g-72$. But there are $g$ twist-bend parameters $\kappa_i= (s_i,\psi_i,\xi_{i_1}, \xi_{i_2}, k_{1i}, k_{2i}, k_{3i})$, one for each handle,  and each contributes $10$ real parameters. Thus, we need  $62g-72 + 10g=72g-72$ real parameters to determine $\rho$ up to conjugacy.

This proves the theorem.

\bigskip 

\begin{acknowledgement}
John Parker suggested us to use flags in the definition of the weakly non-singular pairs. We thank him for comments and  suggestions.  
\end{acknowledgement}

%\bibliographystyle{alpha}
%\bibliography{qpairs.bib}

\begin{thebibliography}{199}

\bibitem[AK07]{ak}
B.~N. Apanasov and I.~Kim.
\newblock Cartan's angular invariant and deformations in symmetric spaces of
  rank 1.
\newblock {\em Mat. Sb.}, 198(2):3--28, 2007.

\bibitem[Cao16]{cao}
Wensheng Cao.
\newblock Congruence classes of points in quaternionic hyperbolic space.
\newblock {\em Geom. Dedicata}, 180:203--228, 2016.

\bibitem[ChGr]{cg}
S.S. {Chen} and L.~{Greenberg}.
\newblock {Hyperbolic spaces.}
\newblock {Contribut. to Analysis, Collect. of  papers dedicated to Lipman Bers,
  49--87}, 1974.

\bibitem[CG10]{cugu2}
Heleno Cunha and Nikolay Gusevskii.
\newblock On the moduli space of quadruples of points in the boundary of
  complex hyperbolic space.
\newblock {\em Transform. Groups}, 15(2):261--283, 2010.

\bibitem[CG12]{cugu1}
Heleno Cunha and Nikolay Gusevskii.
\newblock The moduli space of points in the boundary of complex hyperbolic
  space.
\newblock {\em J. Geom. Anal.}, 22(1):1--11, 2012.

\bibitem[GJ17]{jg}
Gaoshun {Gou} and Yueping {Jiang}.
\newblock {The moduli space of points in the boundary of quaternionic
  hyperbolic space.}
\newblock {\em  Osaka J. Math.}, 57(4): 827--846, 2020. 
\bibitem[GK1]{gk1}
Krishnendu {Gongopadhyay} and Sagar~B. {Kalane}.
\newblock {Quaternionic Hyperbolic {F}enchel-{N}ielsen coordinates.}
\newblock {\em {Geom. Dedicata}}, 199(1):247--271, 2019.

\bibitem[GK2]{gk2}
Krishnendu {Gongopadhyay} and Sagar~B. {Kalane}.
\newblock { Conjugation orbits of semisimple pairs in rank one.}
\newblock {\em {Forum. Math.}}, 31(5):1097--1118, 2019.

\bibitem[GL17]{gl}
Krishnendu Gongopadhyay and Sean Lawton.
\newblock Invariants of pairs in {${\rm SL}(4,{\mathbb C})$} and {${\rm
  SU}(3,1)$}.
\newblock {\em Proc. Amer. Math. Soc.}, 145(11):4703--4715, 2017.

\bibitem[Gol99]{gold}
William~M. Goldman.
\newblock {\em Complex hyperbolic geometry}.
\newblock Oxford Mathematical Monographs. The Clarendon Press, Oxford
  University Press, New York, 1999.
\newblock Oxford Science Publications.

\bibitem[Gon13]{kgz}
Krishnendu Gongopadhyay.
\newblock The {$z$}-classes of quaternionic hyperbolic isometries.
\newblock {\em J. Group Theory}, 16(6):941--964, 2013.

\bibitem[GP13]{gop}
Krishnendu Gongopadhyay and Shiv Parsad.
\newblock Classification of quaternionic hyperbolic isometries.
\newblock {\em Conform. Geom. Dyn.}, 17:68--76, 2013.

\bibitem[GP18A]{gp2}
Krishnendu Gongopadhyay and Shiv Parsad.
\newblock On {F}enchel-{N}ielsen coordinates of surface group representations
  into {${\rm SU}(3,1)$}.
\newblock {\em Math. Proc. Cambridge Phil. Soc.}, 165(1):1--23, 2018.

\bibitem[GP18B]{gp3}
Krishnendu {Gongopadhyay} and Shiv {Parsad}.
\newblock {Conjugation orbits of loxodromic pairs in $\mathrm{SU}(n,1)$.}
\newblock {\em {Bull. Sci. Math}}, 148:14--32, 2018.


\bibitem[KR87]{kr}
A.~Kor\'{a}nyi and H.~M. Reimann.
\newblock The complex cross ratio on the {H}eisenberg group.
\newblock {\em Enseign. Math. (2)}, 33(3-4):291--300, 1987.

\bibitem[MW12]{mw}
Julien March\'e and Pierre Will.
\newblock Configuration of flags and representations of surface groups in complex hyperbolic geometry. 
\newblock {\em Geom. Dedicata}, 156:49--70, 2012.

\bibitem[Pla14]{platis}
Ioannis~D. Platis.
\newblock Cross-ratios and the {P}tolemaean inequality in boundaries of
  symmetric spaces of rank 1.
\newblock {\em Geom. Dedicata}, 169:187--208, 2014.

\bibitem[PP08]{pp}
John~R. Parker and Ioannis~D. Platis.
\newblock Complex hyperbolic {F}enchel-{N}ielsen coordinates.
\newblock {\em Topology}, 47(2):101--135, 2008.

\end{thebibliography}

\end{document}